\documentclass[12pt]{amsart}
\input xy  \xyoption{all}
\usepackage{enumerate}
\usepackage{dsfont}

\def\tt{{\boldsymbol t}}
\def\zz{{\boldsymbol z}}
\DeclareMathOperator\Sym{Sym}

\def\Wt{\tilde{W}}
\def\II{{\mathcal I}}
\def\Il{\II_\lambda}
\def\ep{\epsilon}

\DeclareMathOperator\One{\mathds 1}
\DeclareMathOperator\C{\mathbb C}

\DeclareMathOperator\Fl{{\mathcal F}_\lambda}
\DeclareMathOperator\spa{span}

\newtheorem{fact}{Fact}[section]

\newtheorem{theorem}[fact]{Theorem}
\newtheorem{definition}[fact]{Definition}
\newtheorem{example}[fact]{Example}
\newtheorem{rremark}[fact]{Remark}
\newenvironment{remark}{\begin{rremark} \rm}{\end{rremark}}

\newtheorem{corollary}[fact]{Corollary}

\title[Equivariant CSM classes in partial flag varieties]{Equivariant Chern-Schwartz-MacPherson classes in partial flag varieties: interpolation and formulae}

\author{R. Rim\'anyi}
\address{Department of Mathematics, University of North Carolina at Chapel Hill, USA}
\email{rimanyi@email.unc.edu}

\author{A. Varchenko}
\address{Department of Mathematics, University of North Carolina at Chapel Hill, USA}
\email{anv@email.unc.edu}

\thanks{R.R. was supported by NSF grant DMS-1200685 and A.V.
was supported in part by NSF grant DMS-1362924,  Simons Foundation, and the Max Planck Institute in Bonn.}

\begin{document}

\dedicatory{Dedicated to Piotr Pragacz on the occasion
     of his $60$th birthday}

\begin{abstract}
Consider the natural torus action on a partial flag manifold $\Fl$. Let $\Omega_I\subset \Fl$ be an open Schubert variety, and let $c^{sm}(\Omega_I)\in H_T^*(\Fl)$ be its torus equivariant Chern-Schwartz-MacPherson class. We show a set of interpolation properties that uniquely determine $c^{sm}(\Omega_I)$, as well as a formula, of `localization type', for $c^{sm}(\Omega_I)$. In fact, we proved similar results for a class $\kappa_I\in H_T^*(\Fl)$ --- in the context of quantum group actions on the equivariant cohomology groups of partial flag varieties. In this note we show that $c^{SM}(\Omega_I)=\kappa_I$.
\end{abstract}

\maketitle

\section{Introduction}
An interesting chapter of enumerative geometry is the theory of characteristic classes of singular varieties. One of the fundamental results about Chern-Schwartz-MacPherson (CSM) characteristic classes is their calculation for degeneracy loci in \cite{PP} by Pragacz and Parusi\'nski. In this short note---dedicated to the 60th birthday of P. Pragacz---we prove a result about CSM classes of Schubert cells in partial flag varieties.

Namely, we present a set of interpolation properties that uniquely determine the sought CSM class. Such interpolation characterisation was known before for the leading term of the CSM class, the fundamental class  \cite{FR}. A solution of these interpolation conditions is a weight function in the terminology of our earlier works. Weight functions were considered by Tarasov and Varchenko \cite{TV1, TV2} in the context of q-hypergeometric solutions of qKZ differential equations, and turn up in our recent works (with Tarasov, Gorbounov) in the context of geometric interpretations of Bethe algebras.

In fact, our results on the interpolation conditions and their solution are present in \cite{RTV1, RTV2, RTV3}  for certain cohomology classes $\kappa_I$---see also one of the main motivations \cite[Section 3.3.4]{MO}. In the present note we only prove that $\kappa_I$ is equal to the CSM class of a Schubert cell $\Omega_I$. By giving a proof of this fact and by presenting an accessible description of the interpolation conditions and the weight functions (with appropriate convention changes) we hope to bring the attention of researchers in Schubert calculus and characteristic classes to the quantum group aspects of CSM classes.

In Section \ref{sec 8}  we extend this result to the case of
   Chern-Schwartz-MacPherson classes of Schubert cells in $G/P$
   where $P$ is a parabolic subgroup of a semisimple group  $G$. 
\smallskip

The authors thank P. Aluffi, L. M. Feh\'er, and T. Ohmoto for useful discussions on the topic.

\section{Weight functions} \label{sec:weightfunctions}

\subsection{Definition of weight functions}

Let us fix non-negative integers $n$ and $N$, as well as $\lambda=(\lambda_1,\ldots, \lambda_N) \in \{0,1,2,\ldots\}^N$ with $|\lambda|=\sum_{k=1}^N \lambda_k=n$.
Consider $N$-tuples $I=(I_1,\ldots,I_N)$, where $I_k\subset \{1,\ldots,n\}$, $|I_k|=\lambda_k$, and $I_k\cap I_l=\emptyset$ for $k\not= l$. The set of such $N$-tuples will be denoted by $\Il$.
For $I\in \Il$ we will use the notations
\[
I^{(k)}=\bigcup_{l=1}^k I_l=\{i^{(k)}_1< i^{(k)}_2 < \ldots < i^{(k)}_{\lambda^{(k)}} \},
\qquad
\lambda^{(k)}=|I^{(k)}|=\sum_{l=1}^k \lambda_l
\]
and the variables
\begin{equation}\label{eqn:vars}
\tt=\{
t^{(k)}_a: k=1,\ldots, N-1, a=1,\ldots,\lambda^{(k)}\},
\qquad
\zz=\{z_1,\ldots,z_n\}.
\end{equation}

For $I\in \Il$, $k=1,\ldots,N-1$, $a=1,\ldots,\lambda^{(k)}$, $b=1,\ldots,\lambda^{(k+1)}$ define
\[
  \ell^{(k)}_I(a,b) = \begin{cases}
    1+t^{(k+1)}_b- t^{(k)}_a                          & \text{if } i^{(k+1)}_b<i^{(k)}_a \\
    \ \ \ \ \  t^{(k+1)}_b- t^{(k)}_a                             & \text{if } i^{(k+1)}_b>i^{(k)}_a, \\
                             \end{cases}
\]
with the convention that $t^{(N)}_b=z_b$.

Let the group $S^{(\lambda)}=S_{\lambda^{(1)}} \times \ldots \times S_{\lambda^{(N-1)}} $ act on the set of variables $\tt$ in such a way that elements of $S_{\lambda^{(k)}}$ permute the lower indexes of $t^{(k)}_a$'s. Let $\Sym_{\lambda} f(\tt)=\sum_{\sigma\in S^{(\lambda)}} f( \sigma(\tt) ).$

Denote $e_\lambda(\tt)=\prod_{k=1}^{N-1} \prod_{a=1}^{\lambda^{(k)}}\prod_{b=1}^{\lambda^{(k)}} (1+t^{(k)}_b-t^{(k)}_a)$.

\begin{definition}
Define the weight function
\[
W_I(\tt,\zz)=
\Sym_{\lambda}
     \left[
         \prod_{k=1}^{N-1} \left(
            \prod_{a=1}^{\lambda^{(k)}} \prod_{b=1}^{\lambda^{(k+1)}} \ell^{(k)}_I(a,b) \cdot
           \prod_{a=1}^{\lambda^{(k)}}  \prod_{b=a+1}^{\lambda^{(k)}}
              \frac{1+t^{(k)}_b-t^{(k)}_a}{t^{(k)}_b-t^{(k)}_a}
        \right)
\right],
\]
and the modified weight function
\[
\Wt_I(\tt,\zz)=W_I(\tt,\zz)/e_\lambda(\tt).
\]
\end{definition}

The weight functions $W_I$ are polynomials (despite the appearance of $t^{(k)}_b-t^{(k)}_a$ factors in the denominators), while the modified weight functions $\Wt_I$ are rational functions.

\begin{example}
For $N=2$, $\lambda=(1,n-1)$ we have
\[
W_{(\{k\}, \{1,\ldots,n\}-\{k\})}= \prod_{i=1}^{k-1} (1+z_i-t^{(1)}_1) \cdot \prod_{i=k+1}^n (z_i- t^{(1)}_1).
\]
\end{example}

\section{Combinatorial description of weight function}
\label{sec comb}

In this section we show a diagrammatic interpretation of the terms of the weight function. Let $I\in\Il$. Consider a table with $n$ rows
and $N$ columns. Number the rows from top to bottom and number the columns from
left to right. Certain boxes of this table will be distinguished, as follows.
In the $k$'th column distinguish boxes in the $i$'th row if $i\in I^{(k)}$.
This way all the boxes in the last column will be distinguished since
$I^{(N)}=\{1,\ldots, n\}$.

Now we will define fillings of the tables by putting various variables in
the distinguished boxes. First, put the variables $z_1,\ldots, z_n$ into the last
column from top to bottom. Now choose permutations
$\sigma_1\in S_{\lambda^{(1)}},  \ldots,
\sigma_{N-1} \in S_{\lambda^{(N-1)}}$.
Put the variables $t^{(k)}_{\sigma_k(1)}, \ldots,  t^{(k)}_{\sigma_k(\lambda^{(k)})}$
in the distinguished boxes of the $k$'th column from top to bottom.

Each such filled table will define a rational function as follows. Let $u$ be
a variable in the filled table in one of the columns $1,\ldots,N-1$. If $v$ is
a variable in the next column, but above the position of $u$ then consider
the factor $1+v-u$ (`type-1 factor'). If $v$ is a variable in the next column,
but below the position of $u$ then consider the factor $v-u$ (`type-2
factor'). If $v$ is a variable in the same column, but below the position
of $u$ then consider the factor $(1+v-u)/(v-u)$ (`type-3 factor').
The rule is illustrated in the following figure.
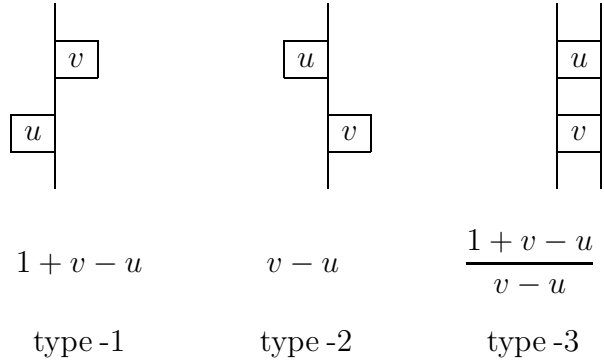
\[
\begin{tabular}{c|c}
& \\
\cline{2-2}
& \multicolumn{1}{|c|}{$v$} \\
\cline{2-2}
& \\
\cline{1-1}
\multicolumn{1}{|c|}{$u$} & \\
\cline{1-1}
&
\end{tabular} \qquad\qquad\qquad
\begin{tabular}{c|c}
& \\
\cline{1-1}
\multicolumn{1}{|c|}{$u$} & \\
\cline{1-1}
& \\
\cline{2-2}
& \multicolumn{1}{|c|}{$v$} \\
\cline{2-2}
&
\end{tabular} \qquad\qquad\qquad
\begin{tabular}{|c|}
\\
\cline{1-1}
$u$ \\
\cline{1-1}
\\
\cline{1-1}
$v$ \\
\cline{1-1}
\\
\end{tabular}
\]
\[1+v-u \qquad\qquad v-u \qquad\qquad \frac{1+v-u}{v-u}\]
\[\text{type\:-1} \qquad\qquad\ \text{type\:-2} \qquad\qquad\ \text{type\:-3}\]
For each variable $u$ in the table consider all these factors and multiply them together. This is ``the term associated with the filled table''.

One sees that $W_I$ is the sum of terms associated with the filled tables corresponding to all choices $\sigma_1,\ldots, \sigma_{N-1}$. For example, $W_{\{2\},\{1\},\{3\}}$ is the sum of two terms associated with the filled tables
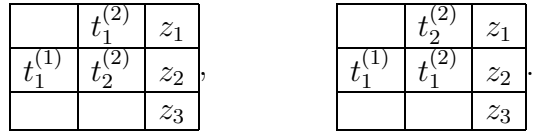
\[
\begin{tabular}{|c|c|c|}
\hline
& $t^{(2)}_1$ & $z_1$ \\
\hline
$t^{(1)}_1$ & $t^{(2)}_2$ & $z_2$ \\
\hline
& & $z_3$ \\
\hline
\end{tabular},\qquad\qquad
\begin{tabular}{|c|c|c|}
\hline
& $t^{(2)}_2$ & $z_1$ \\
\hline
$t^{(1)}_1$ & $t^{(2)}_1$ & $z_2$ \\
\hline
& & $z_3$ \\
\hline
\end{tabular}.
\]
The term corresponding to the first filled table is
\[
\underbrace{(1+t^{(2)}_1-t^{(1)}_1)(1+z_1-t^{(2)}_2)}_{type-1}\,
\underbrace{(z_2-t^{(2)}_1)(z_3-t^{(2)}_1)(z_3-t^{(2)}_2)}_{type-2}\,
\underbrace{\frac{(1+t^{(2)}_2-t^{(2)}_1)}{(t^{(2)}_2-t^{(2)}_1)} }_{type-3},
\]
and the term corresponding to the second filled table is
\[
\underbrace{(1+t^{(2)}_2-t^{(1)}_1)\>(1+z_1-t^{(2)}_1)}_{type-1}\,
\underbrace{(z_2-t^{(2)}_2)\>(z_3-t^{(2)}_2)\>(z_3-t^{(2)}_1)}_{type-2}\,
\underbrace{\frac{(1+t^{(2)}_1-t^{(2)}_2)}{(t^{(2)}_1-t^{(2)}_2)} }_{type-3}.
\]

\section{Partial flag manifold, equivariant cohomology, Schubert varieties}
\label{sec:flag}

Let $\ep_1,\ldots, \ep_n$ be the standard basis in $\C^n$. Consider the partial flag manifold $\Fl$ parameterizing chains
\[
0=F_0 \subset F_1 \subset \ldots \subset F_N=\C^n,
\]
where $\dim F_i/F_{i-1}=\lambda_i$. The standard action of the torus $T=(\C^*)^n$ on $\C^n$ induces an action of $T$ on $\Fl$. The fixed points of this action are the points $x_I=(F_i)$ with $F_k=\spa\{ \ep_i: i\in I^{(k)}\}$.

Consider variables $\Gamma_i=\{\gamma_{i,1},\ldots,\gamma_{i,\lambda_i}\}$ for $i=1,\ldots,N$, and let $\Gamma=\{\Gamma_1,\ldots,\Gamma_N\}$. The group $S_\lambda=\times_i S_{\lambda_i}$ acts on $\Gamma$ by permuting the variables with the same first index. The complex coefficient, $T$ equivariant cohomology ring of $\Fl$ is presented as
\begin{equation}\label{eqn:presentation}
H_T^*(\Fl)=
\C[\Gamma]^{S_\lambda}\otimes \C[\zz] \ /\  \langle f(\Gamma)=f(\zz) \ \text{for any}\ f\in\C[\zz]^{S_n} \rangle.
\end{equation}
Here $\gamma_{i,j}$ for $j=1,\ldots,\lambda_i$ are the Chern roots of the bundle whose fiber is $F_i/F_{i-1}$, and $\zz$ are the Chern roots of the torus.

Let $V_i=\spa(\ep_1,\ldots,\ep_i)$. For $I\in \Il$ define the Schubert cell
\[
\Omega_I=\{ F\in \Fl\ | \ \dim(F_p\cap V_q)=\#\{i\in I^{(p)}|i\leq q\} \ \forall p\leq N \ \forall q\leq n\}.
\]
The Schubert cell $\Omega_I$ is an affine space of dimension
\[
\#\{(i,j)\in \{1,\ldots,n\}^2\ |\ i\in I_a, j\in I_b, a<b, i>j\}.
\]

The point $x_I$ is a smooth point of the Schubert cell $\Omega_I$. The weights of the torus action on the tangent space $T_{x_I}\Omega_I$ are
\[    z_b-z_a \qquad \text{for} \qquad    a>b,\ a\in I_k, \ b\in I_l,  \ k<l. \]
The weights of the torus action on a $T$ invariant normal space $N_{x_I}\Omega_I$ to $T_{x_I}\Omega_I$ in $T_{x_I}\Fl$ are
\[    z_b-z_a \qquad \text{for} \qquad   a<b,\  a\in I_k, \ b\in I_l, \  k<l. \]
Hence we have
\[
c(T_{x_I}\Omega_I)=\prod_{k<l} \prod_{a\in I_k}  \mathop{\prod_{b \in I_l}}_{a>b}    (1+z_b-z_a),\qquad
e(N_{x_I}\Omega_I )=\prod_{k<l} \prod_{a\in I_k}  \mathop{\prod_{b \in I_l}}_{b>a}        (z_b-z_a),\qquad
\]
for the tangent total Chern class and the normal Euler class of $\Omega_I$ at $x_I$.

\section{Unique classes defined by interpolation}
\label{sec:axiom}

The restriction of a cohomology class $\omega \in H_T^*(\Fl)$ to the fixed point $x_J$ will be denoted $\omega|_{x_J}$. In terms of variables this means the substitution
\begin{equation}\label{eqn:rest}
\{\gamma_{k,i}\ |\ i=1,\ldots,\lambda_k\} \mapsto \{z_a\ |\  a\in J_k\}.
\end{equation}

For an $S^{(\lambda)}$ symmetric function $f(\tt,\zz)$ in variables $\tt, \zz$ as in (\ref{eqn:vars}) let $f(\Gamma,\zz)$ denote the substitution
\[
\{ t^{(k)}_a \ |\  a=1\ldots, \lambda^{(k)} \}
\mapsto
\{ \Gamma_1, \ldots, \Gamma_k \}.
\]

\begin{theorem}\label{thm:axiom}
There are unique classes $\kappa_I \in H_T^*(\Fl)$ satisfying
\begin{enumerate}[(I)]
\item \label{ax1} $\kappa_I|_{x_J}$ is divisible by $c(T_{x_J}\Omega_J)$;
\item \label{ax2} $\kappa_I|_{x_I}=c(T_{x_I}\Omega_I)e(N_{x_I}\Omega_I)$;
\item \label{ax3} $\kappa_I|_{x_J}$ has degree less than $\dim \Fl=\sum_{k<l} \lambda_k \lambda_l$, if $I\not= J$.
\end{enumerate}
Moreover,
\begin{equation}\label{eqn:kW}
\kappa_I=[\tilde{W}_I(\Gamma,\zz)].
\end{equation}
\end{theorem}

Since $\Wt_I$ is a rational function we need to explain what we mean by (\ref{eqn:kW}). This means that the substitution (\ref{eqn:rest}) of $\tilde{W}_I(\Gamma,\zz)$ is (a polynomial, and is) equal to $\kappa_I|_{x_J}$ for all $J \in \Il$. It is remarkable that although $\kappa_I$ can be represented by a polynomial in $\Gamma$, $\zz$ variables (by definition, see (\ref{eqn:presentation})) but we represented it by a rational function.

\smallskip

The first part of Theorem \ref{thm:axiom} follows from \cite[Thm 3.3.4]{MO}. Although (\ref{ax1}) is a local version of the axiom in \cite{MO}, the proof is essentially the same. Both parts of Theorem \ref{thm:axiom} follow from the K-theory analogue \cite[Thm 3.1 and Thm 6.7]{RTV3}, see also \cite{RTV1,RTV2}.

\section{Chern-Schwartz-MacPherson classes}

Let the torus $T=(\C^*)^n$ act on the algebraic manifold $M$. For a $T$ equivariant constructible function $f$ on $M$ one can consider its Chern-Schwartz-MacPherson (CSM) class $c^{SM}(f)\in H^*_T(M)$.
For an invariant subset $X\subset M$ we write $c^{SM}(X)$ for $c^{SM}(\One_X)$ where $\One_X$ is the indicator function of $X$.

The non-equivariant version of this notion was introduced by MacPherson \cite{M}, see also \cite{Sch}, and the equivariant version by Ohmoto \cite{O1}, see also \cite{W,O2}. We refer the reader to these papers for definitions and main properties, namely natural normalization, functoriality, product, and localization properties. CSM classes of Schubert cells and varieties, their recursion and positivity properties, are studied in \cite{AM1,H, AM2}.

\begin{remark}
It is customary to consider the CSM class of a variety $X$ in its (equivariant) homology $c_{SM}(X)\in H^T_*(X)$---this version does not depend on an ambient manifold $M$. In the present paper we will not consider this homology CSM class. We follow \cite{O1} and \cite{W} by mapping the homology CSM class $c_{SM}(X)$ first to $H^T_*(M)$ by homology push-forward, then applying (equivariant) Poincar\'e duality for $M$, and considering the resulting cohomology $c^{SM}(X)$ class living in $H_T^*(M)$.
\end{remark}

The class $c^{SM}(X)$ is a refinement of the notion of (equivariant) fundamental class; namely $c^{SM}(X)=[X] +$ higher degree terms, where $[X]$ is the (equivariant) fundamental class of the variety $X$ in $H^*_T(M)$.

Now we recall the properties of equivariant CSM classes we will need in the next section.

\begin{enumerate}[(A)]
\item (Linearity.) For equivariant constructible functions $f$ and $g$ we have $c^{SM}(f+g)=c^{SM}(f)+c^{SM}(g)$.
\item (Local model of CSM classes of varieties.) Suppose $T$ acts on a vector space $V$, and let $X \subset V$ be an invariant smooth subvariety. Then
\begin{equation}\label{eqn:smooth}
c^{SM}(X)= c(T_0X)e(V/T_0X)= \prod_i (1+\alpha_i) \cdot \prod_j \beta_j \qquad \in H^*_T(V)=H^*_T(\{pt\}).
\end{equation}
where $\alpha_i$'s are the tangent weights, and $\beta_j$'s are the normal weights of $X$. Indeed, when $c^{SM}(X)$ is considered in the (co)homology of $X$ then it is the total Chern class of the tangent bundle of $X$ \cite[Thm. 3.10]{O2}. In our convention (i.e. when $c^{SM}(X)$ is pushed forward to the cohomology of the ambient space) we obtain (\ref{eqn:smooth}).

\item \label{item:divisibility}
Suppose $T$ acts on a vector space $V$, and let $X$ be an invariant subvariety which contains an invariant smooth subvariety $X_0$. Moreover assume that an invariant linear subspace $W$ is transversal to $X$ and $W\oplus T_0X_0=V$. Then $c^{SM}(X)$ is divisible by $c(T_0X_0)$.
Indeed, in our convention Theorem 3.13 (2) of \cite{O2} reads
\[ \frac{   c^{SM}(X) }{c(V) } = \frac{ c^{SM}(X \cap W)}{c(W)},\]
which implies $c^{SM}(X)= c(V/W) c^{SM}(X \cap W) = c(T_0X_0) c^{SM}(X \cap W).$
\item  \label{item:W}
Let $X$ be a subvariety of the $T$-manifold $M$ with isolated fixed points. Let $x$ be a fixed point of $M$. Then  \cite[Thm 20]{W}
    \[
    c^{SM}(X)|_x=
    \begin{cases}
       e(T_xM)+\text{lower degree terms} & \text{if} \  x\in X \\
       0 & \text{if}\  x\not\in X.
    \end{cases}
    \]
    A remarkable feature of this fact which we will use below is that the top degree part of $c^{SM}(X)|_x$  does not depend on the variety $X$ as long as $x$ belongs to the variety.
\end{enumerate}

\section{CSM classes coincide with $\kappa$ classes} \label{sec:coincide}

Consider the torus action on $\Fl$ and their Schubert cells $\Omega_I\subset \Fl$ as in Section \ref{sec:flag}. Also recall the $\kappa_I$ classes defined axiomatically in Section \ref{sec:axiom}.

\begin{theorem} \label{thm:ck}
We have
\[
c^{SM}(\Omega_I)=\kappa_I \qquad\qquad \in H^*_T(\Fl).
\]
\end{theorem}

\begin{proof}
We need to prove that the class $c^{SM}(\Omega_I)$ satisfies the axioms defining $\kappa_I$.

Let us write the indicator function of the Schubert cell as a linear combination of the indicator functions of Schubert varieties:
\begin{equation}\label{eqn:indfunct}
\One_{\Omega_I}=\sum_{K \leq I} d_{IK} \One_{\overline{\Omega}_K}.
\end{equation}
Here $K\leq I$ means $\Omega_K \subset \overline{\Omega}_I$ and  $d_{IK}$ are integer coefficients, $d_{II}=1$.
By linearity of CSM classes
\[
c^{SM}(\Omega_I)=\sum_{K\leq I} d_{IK} c^{SM}(\overline{\Omega}_K),
\]
and hence
\[
c^{SM}(\Omega_I)|_{x_J}= \sum_{K \leq I} d_{IK} c^{SM}(\overline{\Omega}_K)|_{x_J} = \sum_{J\leq K \leq I} d_{IK} c^{SM}(\overline{\Omega}_K)|_{x_J}.
\]
The behaviour of the varieties $\overline{\Omega}_K$ appearing on the RHS near $x_J$ are known in Schubert calculus: they all contain $\Omega_J$,  $\Omega_J$ is smooth at $x_J$ and the subspace $W$ required in (\ref{item:divisibility}) above exists. Therefore each term is divisible by $c(T_{x_J}\Omega_J)$. This proves axiom (\ref{ax1}).

If $I=J$ then the right hand side is just the term $c^{SM}(\overline{\Omega}_I)|_{x_I}$ which---according to the ``local model'' property above is $c(T_{x_I}\Omega_I)e(N_{x_I}\Omega_I)$. This proves axiom (\ref{ax2}).

Now let $J \leq I, J\not=I$. Using Theorem \cite[Thm 20]{W} as recalled above in (\ref{item:W}) we obtain
\[
c^{SM}(\Omega_I)|_{x_J}= \sum_{J\leq K \leq I} d_{IK}(e(T_{x_J}\Fl)+l.d.t.) =
\left( e(T_{x_J}\Fl) \cdot \sum_{J\leq K \leq I} d_{IK} \right) + l.d.t. ,
\]
where ``l.d.t.''($=$lower degree terms) means terms of degree strictly less than $\dim \Fl=\sum_{i<j} \lambda_i\lambda_j$. However, the sum $\sum_{J\leq K \leq I} d_{IK}$ vanishes due to substituting $x_J$ into the functional identity (\ref{eqn:indfunct}). We obtain that $c^{SM}(\Omega_I)|_{x_J}$ has degree strictly less than $\dim \Fl$. This proves axiom (\ref{ax3}).
\end{proof}

\begin{corollary}\label{corcor}
The equivariant CSM classes of Schubert cells of partial flag manifolds are defined by the axioms of Theorem \ref{thm:axiom}. Moreover $c^{SM}(\Omega_I)=[\Wt_I(\Gamma,\zz)]$.
\end{corollary}

\begin{remark} 
The CSM classes satisfy obvious vanishing properties respecting the Bruhat order. Namely, if $\Omega_J\not\subset \overline{\Omega}_I$ then $c^{SM}(\Omega_I)|_{x_J}=0$. It is remarkable that one does not need to list this obvious vanishing property among the axioms, this property is a consequence of the axioms.
\end{remark}

Other remarkable properties of $\kappa_I$ classes include orthogonality properties \cite[Sec.5]{RTV2}, \cite[Thm. 6.6]{RTV3}, recursion \cite[Lemma 3.3]{RTV2},\cite[Thm. 6.10]{RTV3}, and R-matrix properties \cite[Cor. 3.8]{RTV2},\cite[Section 7]{RTV3}. According to Theorem \ref{thm:ck} these also hold for equivariant CSM classes of Schubert cells.

\section{CSM classes of Schubert cells in generalized flag manifolds}
\label{sec 8}

Let $G$ be a semisimple algebraic group and let $T\subset B \subset P$ be a maximal torus, Borel subgroup, and a parabolic subgroup in $G$. The $B$ orbits $\Omega_I$ are called Schubert cells in $G/P$. They are parameterized by certain elements of the Weyl group. Each orbit contains a $T$ fixed point $x_I$. The first part of Corollary \ref{corcor} generalizes to the general flag variety $G/P$ as follows.

\begin{theorem}
The $T$ equivariant Chern-Schwartz-MacPherson classes $c^{SM}(\Omega_I)$ are the unique classes in $H_T^*(G/P)$ satisfying the properties
\begin{enumerate}[(I)]
\item \label{pax1} $c^{SM}(\Omega_I)|_{x_J}$ is divisible by $c(T_{x_J}\Omega_J)$;
\item \label{pax2} $c^{SM}(\Omega_I)|_{x_I}=c(T_{x_I}\Omega_I)e(N_{x_I}\Omega_I)$;
\item \label{pax3} $c^{SM}(\Omega_I)|_{x_J}$ has degree less than $\dim G/P$, if $I\not= J$.
\end{enumerate}
\end{theorem}

\begin{proof}
The existence and uniqueness of classes satisfying the itemized properties is proved in \cite{MO}. (A guide to the reader: Maulik and Okounkov prove that certain classes are characterized by three axioms, see \cite[Section 3.3.4]{MO}. Their three axioms for $G/P$ are (\ref{pax2}), (\ref{pax3}) above, together with a global version of (\ref{pax1}). The global versus local versions of (\ref{pax1}) does not affect the arguments. Also, the Maulik-Okounkov classes are homogeneous, depending on an extra parameter $h$. Our non-homogeneous versions are obtained by putting $h=1$.)

The fact that the $c^{SM}(\Omega_I)$ classes satisfy the properties follows the same way as in the special case detailed in Section \ref{sec:coincide}.
\end{proof}

\begin{remark}  The map sending fixed points $\{x_I\}$ to the   
    classes $\{ c^{SM} (\Omega_I) \}$ is essentially the stable envelope map of Maulik-Okounkov
     corresponding to the pair $P\subset G$.
\end{remark}

\end{document}